\documentclass[reqno]{amsart}

\usepackage{amssymb}
\usepackage{graphicx}
\usepackage{amscd}
\usepackage[hidelinks]{hyperref}
\usepackage{color}
\usepackage{float}
\usepackage{graphics,amsmath,amssymb}
\usepackage{amsthm}
\usepackage{amsfonts}
\usepackage{latexsym}
\usepackage{epsf}
\usepackage{xifthen}
\usepackage{mathrsfs}
\usepackage{dsfont}
\usepackage{makecell}
\usepackage[FIGTOPCAP]{subfigure}
\usepackage{amsmath}
\allowdisplaybreaks[4]
\usepackage{listings}
\usepackage{etoolbox}
\usepackage{fancyhdr}
\usepackage{pdflscape}
\usepackage[title,toc,titletoc]{appendix}
\usepackage{enumitem}
\usepackage[noadjust]{cite}
\usepackage{forest}

 \newtheoremstyle{mytheorem}
 {3pt}
 {3pt}
 {\slshape}
 {}
 {\bfseries}
 {.}
 { }
 {}

\numberwithin{equation}{section}

\theoremstyle{theorem}
\newtheorem{theorem}{Theorem}[section]
\newtheorem*{theorem*}{Theorem}

\newtheorem{lemma}[theorem]{Lemma}

\newtheorem{innercustomgeneric}{\customgenericname}
\providecommand{\customgenericname}{}
\newcommand{\newcustomtheorem}[2]{%
	\newenvironment{#1}[1]
	{%
		\renewcommand\customgenericname{#2}%
		\renewcommand\theinnercustomgeneric{##1}%
		\innercustomgeneric
	}
	{\endinnercustomgeneric}
}
\newcustomtheorem{ctheorem}{Theorem}
\newcustomtheorem{clemma}{Lemma}

\theoremstyle{definition}

\newtheorem*{example*}{Example}
\newtheorem*{examples*}{Examples}

\newtheorem*{remark*}{Remark}
\newtheorem*{remarks*}{Remarks}

\newtheoremstyle{named}{}{}{\itshape}{}{\bfseries}{.}{.5em}{#1\thmnote{ #3}}
\theoremstyle{named}

\newcommand{\Keywords}[1]{\ifthenelse{\isempty{#1}}{}{\smallskip \smallskip \noindent \textbf{Keywords}. #1}}
\newcommand{\MSC}[2][2020]{\ifthenelse{\isempty{#2}}{}{\smallskip \smallskip \noindent \textbf{#1MSC}. #2}}
\newcommand{\abstractnote}[1]{\ifthenelse{\isempty{#1}}{}{\smallskip \smallskip \noindent \textsuperscript{\dag}#1}}

\makeatletter
\def\specialsection{\@startsection{section}{1}%
  \z@{\linespacing\@plus\linespacing}{.5\linespacing}%
  {\normalfont}}
\def\section{\@startsection{section}{1}%
  \z@{.7\linespacing\@plus\linespacing}{.5\linespacing}%
  {\normalfont\scshape}}
\patchcmd{\@settitle}{\uppercasenonmath\@title}{\Large\boldmath}{}{}
\patchcmd{\@settitle}{\begin{center}}{\begin{flushleft}}{}{}
\patchcmd{\@settitle}{\end{center}}{\end{flushleft}}{}{}
\patchcmd{\@setauthors}{\MakeUppercase}{\normalsize}{}{}
\patchcmd{\@setauthors}{\centering}{\raggedright}{}{}
\patchcmd{\section}{\scshape}{\large\bfseries\boldmath}{}{}
\patchcmd{\subsection}{\bfseries}{\bfseries\boldmath}{}{}
\renewcommand{\@secnumfont}{\bfseries}
\patchcmd{\@startsection}{\@afterindenttrue}{\@afterindentfalse}{}{}
\patchcmd{\abstract}{\leftmargin3pc}{\leftmargin1pc}{}{}

\def\maketitle{\par
  \@topnum\z@ 
  \@setcopyright
  \thispagestyle{empty}
  \ifx\@empty\shortauthors \let\shortauthors\shorttitle
  \else \andify\shortauthors
  \fi
  \@maketitle@hook
  \begingroup
  \@maketitle
  \toks@\@xp{\shortauthors}\@temptokena\@xp{\shorttitle}%
  \toks4{\def\\{ \ignorespaces}}
  \edef\@tempa{%
    \@nx\markboth{\the\toks4
      \@nx\MakeUppercase{\the\toks@}}{\the\@temptokena}}%
  \@tempa
  \endgroup
  \c@footnote\z@
  \@cleartopmattertags
}
\makeatother

\newcommand{\cJ}{\mathcal{J}}
\newcommand{\cP}{\mathcal{P}}

\newcommand{\qbinom}[2]{{#1\brack #2}}

\setbox0=\hbox{$+$}
\newdimen\plusheight
\plusheight=\ht0
\def\+{\;\lower\plusheight\hbox{$+$}\;}

\setbox0=\hbox{$-$}
\newdimen\minusheight
\minusheight=\ht0
\def\-{\;\lower\minusheight\hbox{$-$}\;}

\setbox0=\hbox{$\cdots$}
\newdimen\cdotsheight
\cdotsheight=\plusheight
\def\cds{\lower\cdotsheight\hbox{$\cdots$}}

\newcommand{\lastddots}{%
	\raisebox{\dimexpr1ex-\height}{%
		$\displaystyle
		\raisebox{.5\height}{$\ddots$}
		$%
	}%
}

\title[Hankel determinants for $q$-Euler numbers]{Hankel determinants and Jacobi continued fractions for $q$-Euler numbers}

\author[S. Chern]{Shane Chern}
\address[S. Chern]{Department of Mathematics and Statistics, Dalhousie University, Halifax, Nova Scotia, B3H 4R2, Canada}
\email{chenxiaohang92@gmail.com}

\author[L. Jiu]{Lin Jiu*}\thanks{*corresponding author}
\address[L. Jiu]{Zu Chongzhi Center for Mathematics and Computational Sciences, Duke Kunshan University, Kunshan, Suzhou, Jiangsu Province, 215316, PR China}
\email{lin.jiu@dukekunshan.edu.cn}

\date{}

\begin{document}

\maketitle

\begin{abstract}

The $q$-analogs of Bernoulli and Euler numbers were introduced by Carlitz. Similar to the recent results on the Hankel determinants for the $q$-Bernoulli numbers established by Chapoton and Zeng, we determine parallel evaluations for the $q$-Euler numbers. It is shown that the associated Favard-type orthogonal polynomials for $q$-Euler numbers are given by a specialization of the big $q$-Jacobi polynomials, thereby leading to their corresponding Jacobi continued fraction expression, which eventually serves as a key to our determinant evaluations.


\MSC{11B68, 11C20, 30B70, 33D45.}
\end{abstract}

\section{Introduction}

A \emph{Hankel matrix} $(M_{i,j})$ is a square matrix with constant skew diagonals, i.e.~$M_{i,j}=M_{i',j'}$ whenever $i+j=i'+j'$. This terminology was introduced by Hermann Hankel, and in recent years exhibits substantial utility in data analysis, ranging from geophysics \cite{GS1977} to signal processing \cite{Str1997}. Letting $\{s_n\}_{n\ge 0}$ be a sequence in a field $\mathbb{K}$, one may define its associated Hankel matrices by $(s_{i+j})_{0\le i,j\le n}$. In the meantime, it is often of significance to evaluate the determinant of these matrices. Such determinants
\begin{align*}
	\underset{{0\le i,j\le n}}{\det} (s_{i+j})=\det\begin{pmatrix}
		s_0 & s_1 & s_2 & \cdots & s_n\\
		s_1 & s_2 & s_3 & \cdots & s_{n+1}\\
		s_2 & s_3 & s_4 & \cdots & s_{n+2}\\
		\vdots & \vdots & \vdots & \ddots & \vdots\\
		s_n & s_{n+1} & s_{n+2} & \cdots & s_{2n}
	\end{pmatrix}
\end{align*}
are usually called the \emph{Hankel determinants} for $\{s_n\}_{n\ge 0}$.

From a number-theoretic perspective, it is a natural game to take the sequence $s_n$ as classic ones of arithmetic meaning. As an example, considering the \emph{Bernoulli numbers} $B_n$ defined by
\begin{align*}
	\sum_{n\ge 0} B_n\frac{t^n}{n!}= \frac{t}{e^t - 1},
\end{align*}
it was shown by Al-Salam and Carlitz \cite[p.~93, eq.~(3.1)]{AC1959} that
\begin{align*}
	\underset{{0\le i,j\le n}}{\det}(B_{i+j})= (-1)^{\binom{n+1}{2}} \prod_{k=1}^n \frac{(k!)^6}{(2k)!(2k+1)!}.
\end{align*}
Another example treats the \emph{Euler numbers} $E_n$ given by
\begin{align*}
	\sum_{n\ge 0} E_n\frac{t^n}{n!} = \frac{2}{e^t + e^{-t}}.
\end{align*}
Al-Salam and Carlitz \cite[p.~93, eq.~(4.2)]{AC1959} also proved that
\begin{align*}
	\underset{{0\le i,j\le n}}{\det}(E_{i+j})= (-1)^{\binom{n+1}{2}} \prod_{k=1}^n (k!)^2.
\end{align*}
More such Hankel determinant evaluations were nicely collected by Krattenthaler in his seminal surveys \cite[Section 2.7]{Kra1999} and \cite[Section 5.4]{Kra2005}.

There are many ways to offer number sequences a generalization, among which polynomialization is one of the simplest approaches. Simply speaking, one may construct a sequence of polynomials so that it reduces to our original number sequence when the argument is specifically chosen. Along this line, we can define \emph{Bernoulli polynomials} $B_n(x)$ and \emph{Euler polynomials} $E_n(x)$ by
\begin{align*}
	\sum_{n\ge 0} B_n(x)\frac{t^n}{n!} &= \frac{te^{xt}}{e^t - 1},\\
	\sum_{n\ge 0} E_n(x)\frac{t^n}{n!} &= \frac{2e^{xt}}{e^t + 1}.
\end{align*}
Then $B_n = B_n(0)$ and $E_n = 2^n E_n(\frac{1}{2})$. However, the Hankel determinant evaluations for $B_n(x)$ and $E_n(x)$ are not of much excitement as
\begin{align*}
	B_n(x) &= \sum_{k=0}^n \binom{n}{k}B_k x^{n-k},\\
	E_n(x) &= \sum_{k=0}^n \binom{n}{k}\frac{E_k}{2^k} x^{n-k}.
\end{align*}
It is, on the other hand, a standard result \cite[p.~419, Item 445]{MM1960} that
\begin{align*}
	\underset{{0\le i,j\le n}}{\det} (s_{i+j}) = \underset{{0\le i,j\le n}}{\det} \big(s_{i+j}(x)\big),
\end{align*}
where
\begin{align*}
	s_n(x) = \sum_{k=0}^n \binom{n}{k}s_k x^{n-k}.
\end{align*}
Thus,
\begin{align}
	\underset{{0\le i,j\le n}}{\det}\big(B_{i+j}(x)\big) &= (-1)^{\binom{n+1}{2}} \prod_{k=1}^n \frac{(k!)^6}{(2k)!(2k+1)!},\\
	\underset{{0\le i,j\le n}}{\det}\big(E_{i+j}(x)\big) &= (-\tfrac{1}{4})^{\binom{n+1}{2}} \prod_{k=1}^n (k!)^2.\label{eq:Euler-poly-det}
\end{align}
(See, e.g., \cite[p.~94, eqs.~(5.1) and (5.2)]{AC1959}.) Although the above evaluations are independent of the variable $x$, the story becomes different when subsequences of Bernoulli or Euler polynomials are taken into account. As shown in a recent paper of Dilcher and Jiu \cite[Theorem 1.1 and Corollary 5.2]{DJ2021},
\begin{align*}
	\underset{{0\le i,j\le n}}{\det}\big(B_{2(i+j)+1}(\tfrac{x+1}{2})\big) &= (-1)^{\binom{n+1}{2}} \left(\frac{x}{2}\right)^{n+1} \prod_{k=1}^n \left(\frac{k^4 (x^2 - k^2)}{4(2k-1)(2k+1)}\right)^{n-k+1},\\
	\underset{{0\le i,j\le n}}{\det}\big(E_{2(i+j)+1}(\tfrac{x+1}{2})\big) &= (-1)^{\binom{n+1}{2}} \left(\frac{x}{2}\right)^{n+1} \prod_{k=1}^n \left(\frac{k^2 (x^2 - 4k^2)}{4}\right)^{n-k+1}.
\end{align*}
In addition, results of like nature can be found in \cite{DJ2022}.

For other generalizations of number sequences, the $q$-version is sometimes of more ``$q$-riosity.'' Usually, we have a sequence of rational functions in $q$, which yields our original sequence at the limiting case $q\to 1$. Throughout, we introduce the \emph{$q$-integers} for $m\in\mathbb{Z}$,
\begin{align*}
	[m]_q:= \frac{1-q^m}{1-q},
\end{align*}
and the \emph{$q$-factorials} for $M\in\mathbb{N}$,
\begin{align*}
	[M]_q! := \prod_{m=1}^M [m]_q.
\end{align*}
We also require the \emph{$q$-Pochhammer symbols} for $N\in\mathbb{N}\cup\{\infty\}$,
\begin{align*}
	(A;q)_N:=\prod_{k=0}^{N-1} (1-A q^k),
\end{align*}
and
\begin{align*}
	(A,B,\ldots,C;q)_N := (A;q)_N(B;q)_N\cdots (C;q)_N.
\end{align*}
The $q$-analogs of Bernoulli and Euler numbers were introduced by Carlitz \cite{Car1948}:
\begin{align}
	\beta_n &:= \frac{1}{(1-q)^n}\sum_{k=0}^n (-1)^k \binom{n}{k} \frac{k+1}{[k+1]_q},\\
	\epsilon_n &:=\frac{1}{(1-q)^n}\sum_{k=0}^n (-1)^k \binom{n}{k} \frac{1+q}{1+q^{k+1}}.
\end{align}
Alternatively, Carlitz's \emph{$q$-Bernoulli numbers} $\beta_n$ can be recursively defined by $\beta_0=1$ and for $n\ge 1$,
\begin{align}
	\sum_{k=0}^n \binom{n}{k} q^{k+1} \beta_k -\beta_n = \begin{cases}
		1, & \text{if $n=1$};\\
		0, & \text{if $n\ge 2$},
	\end{cases}
\end{align}
while the \emph{$q$-Euler numbers} are given by $\epsilon_0=1$ and recursively for $n\ge 1$,
\begin{align}\label{eq:q-Euler-rec}
	\sum_{k=0}^n \binom{n}{k} q^{k+1} \epsilon_k +\epsilon_n = 0.
\end{align}
It should be remarked that at the $q\to 1$ limit, $\epsilon_n$ reduces to
\begin{align*}
	1, -\tfrac{1}{2}, 0, \tfrac{1}{4}, 0, -\tfrac{1}{2}, 0, \tfrac{17}{8}, 0, -\tfrac{31}{2}, \ldots,
\end{align*}
which is identical to $E_n(0)$ rather than $2^n E_n(\frac{1}{2})$, the Euler numbers $E_n$.

In regard to the Hankel determinants, surprisingly, we still have neat evaluations for the $q$-Bernoulli numbers. It was obtained by Chapoton and Zeng \cite[p.~359, eq.~(4.7)]{CZ2017} that
\begin{align*}
	\underset{{0\le i,j\le n}}{\det}(\beta_{i+j})= (-1)^{\binom{n+1}{2}} q^{\binom{n+1}{3}} \prod_{k=1}^n \frac{([k]_q!)^6}{[2k]_q![2k+1]_q!}.
\end{align*}
Meanwhile, Chapoton and Zeng also evaluated $\det(\beta_{i+j+\ell})_{0\le i,j\le n}$ with $\ell\in\{1,2,3\}$. Now one may naturally ask if similar results exist for the $q$-Euler numbers. Our objective in this paper is to answer this question in the affirmative.

\begin{theorem}\label{th:q-Euler-HDet}
	\begin{align}
		&\underset{{0\le i,j\le n}}{\det}(\epsilon_{i+j})\notag\\
		&\quad = \frac{(-1)^{\binom{n+1}{2}}q^{\frac{1}{4}\binom{2n+2}{3}}}{(1-q)^{n(n+1)}} \prod_{k=1}^n \frac{(q^2,q^2;q^2)_k}{(-q,-q^2,-q^2,-q^3;q^2)_k},\label{eq:Hankel:+0}\\
		&\underset{{0\le i,j\le n}}{\det}(\epsilon_{i+j+1})\notag\\
		&\quad = \frac{(-1)^{\binom{n+2}{2}}q^{\frac{1}{4}\binom{2n+4}{3}}}{(1-q)^{n(n+1)} (1+q^2)^{n+1}} \prod_{k=1}^n \frac{(q^2,q^4;q^2)_k}{(-q^2,-q^3,-q^3,-q^4;q^2)_k},\label{eq:Hankel:+1}\\
		&\underset{{0\le i,j\le n}}{\det}(\epsilon_{i+j+2})\notag\\
		&\quad = \frac{(-1)^{\binom{n+2}{2}}q^{\frac{1}{4}\binom{2n+4}{3}} (1+q)^n \big(1-(-1)^n q^{(n+2)^2}\big)}{(1-q)^{n(n+1)} (1+q^2)^{2(n+1)} (1+q^3)^{n+1}} \prod_{k=1}^n \frac{(q^4,q^4;q^2)_k}{(-q^3,-q^4,-q^4,-q^5;q^2)_k}.\label{eq:Hankel:+2}
	\end{align}
\end{theorem}

\begin{remark*}
	Taking $q\to 1$, \eqref{eq:Hankel:+0} becomes
	\begin{align*}
		\lim_{q\to 1} \underset{{0\le i,j\le n}}{\det}(\epsilon_{i+j}) = (-\tfrac{1}{16})^{\binom{n+1}{2}} \prod_{k=1}^n \big((2k)!!\big)^2 = (-\tfrac{1}{4})^{\binom{n+1}{2}} \prod_{k=1}^n (k!)^2.
	\end{align*}
	In light of the fact that $\lim_{q\to 1}\epsilon_n = E_n(0)$, the above relation matches \eqref{eq:Euler-poly-det} with $x=0$.
\end{remark*}

It is well-known that the evaluation of Hankel determinants is closely related to orthogonal polynomials and continued fractions, so several preliminary lemmas are provided in Section \ref{sec:pre}. Then in Section \ref{sec:big-q-Jacobi}, we introduce the big $q$-Jacobi polynomials, whose orthogonality will be used for our determinant calculations. In addition, to clearly characterize the orthogonality of the big $q$-Jacobi polynomials, we require a linear functional on $\mathbb{Q}(q)[z]$, as discussed in Section \ref{sec:lin-func}. With the above preparations, Section \ref{sec:CF} is devoted to the continued fraction expressions for the generating functions of $\{\epsilon_{k}\}_{k\ge 0}$ and $\{\epsilon_{k+1}\}_{k\ge 0}$, thereby leading to the proof of Theorem \ref{th:q-Euler-HDet} in Section \ref{sec:HDet}. Finally, in Section \ref{sec:con}, we close this paper with some additional discussions.

\section{Preliminaries}\label{sec:pre}

A family of polynomials $\{p_n(z)\}_{n\ge 0}$ with $p_n(z)$ of degree $n$ is called \emph{orthogonal} if there is a linear functional $L$ on the space of polynomials in $z$ such that $L\big(p_m(z)p_n(z)\big) = \delta_{m,n} \sigma_n$ where $\delta_{m,n}$ is the Kronecker delta and $\{\sigma_n\}_{n\ge 0}$ is a fixed nonzero sequence. Notably, orthogonal polynomials can be characterized by Favard's Theorem; see \cite[p.~195, Theorem 50.1]{Wal1948} or \cite[p.~21, Theorem 12]{Kra1999}.

\begin{lemma}[Favard's Theorem]\label{le:Favard}
	Let $\{p_n(z)\}_{n\ge 0}$ be a family of monic polynomials with $p_n(z)$ of degree $n$. Then they are orthogonal if and only if there exist sequences $\{a_n\}_{n\ge 0}$ and $\{b_n\}_{n\ge 1}$ with $b_n\ne 0$ such that $p_0(z)=1$, $p_1(z)=a_0+z$, and for $n\ge 1$,
	\begin{align}\label{eq:3TermRec}
		p_{n+1}(z) = (a_n + z)p_n(z) - b_n p_{n-1}(z).
	\end{align}
\end{lemma}

Note that $\{z^k\}_{k\ge 0}$ forms a basis of the space of polynomials in $z$. Therefore, for a family of orthogonal polynomials, to explicitly express its associated linear functional $L$, it suffices to evaluate $L(z^k)$ for each $k\ge 0$. Such evaluations have a surprising connection with \emph{Jacobi continued fractions}, or \emph{$J$-fractions} for short.

\begin{lemma}[Cf.~{\cite[p.~197, Theorem 51.1]{Wal1948}} or {\cite[p.~21, Theorem 13]{Kra1999}}]\label{le:linear-func}
	Let $L$ be an associated linear functional for a family of orthogonal monic polynomials $\{p_n(z)\}_{n\ge 0}$ with $p_n(z)$ of degree $n$. Then
	\begin{align}
		\sum_{k\ge 0} L(z^k) x^k = \cfrac{L(z^0)}{1 + a_{0} x -\cfrac{b_{1} x^2}{1 + a_{1} x - \cfrac{b_{2} x^2}{1 + a_{2} x - \lastddots}}},
	\end{align}
	where $\{a_n\}_{n\ge 0}$ and $\{b_n\}_{n\ge 1}$ are given by Favard's Theorem, as in \eqref{eq:3TermRec}.
\end{lemma}

Finally, given a $J$-fraction, Heilermann established evaluations of the Hankel determinants for the coefficients in the series expansion of this continued fraction.

\begin{lemma}[Cf.~{\cite[pp.~115--116, Theorem 29]{Kra2005}}]\label{le:det-cf-op}
	Let $\{\mu_k\}_{k\ge 0}$ be a sequence such that its generating function $\sum_{k\ge 0}\mu_k x^k$ has the $J$-fraction expression
	\begin{align*}
		\sum_{k\ge 0}\mu_{k} x^k = \cfrac{\mu_0}{1 + a_{0} x -\cfrac{b_{1} x^2}{1 + a_{1} x - \cfrac{b_{2} x^2}{1 + a_{2} x - \lastddots}}}.
	\end{align*}
	Then for $n\ge 0$,
	\begin{align}\label{eq:det-general-0}
		\underset{{0\le i,j\le n}}{\det}(\mu_{i+j}) = \mu_0^{n+1} b_1^{n} b_2^{n-1} \cdots b_{n-1}^2 b_n.
	\end{align}
	Further, with $a_n$ and $b_n$ from the J-fraction above, define $\{p_n(z)\}_{n\ge 0}$ a family of polynomials given by a three-term recursive relation for $n\ge 1$, 
	\begin{align*}
		p_{n+1}(z) = (a_n + z)p_n(z) - b_n p_{n-1}(z),
	\end{align*}
	with initial conditions $p_0(z)=1$ and $p_1(z)=a_0+z$. Then
	\begin{align}\label{eq:det-general-1}
		\underset{{0\le i,j\le n}}{\det}(\mu_{i+j+1}) = \underset{{0\le i,j\le n}}{\det}(\mu_{i+j})\cdot (-1)^{n+1} p_{n+1}(0).
	\end{align}
\end{lemma}

\begin{remark*}
	It is asserted by Favard's Theorem that the polynomials $p_n(z)$ in Lemma \ref{le:det-cf-op} are orthogonal. Hence, we shall call $p_n(z)$ the \emph{Favard orthogonal polynomials} for $\{\mu_k\}_{k\ge 0}$ throughout.
\end{remark*}

\section{Big $q$-Jacobi polynomials}\label{sec:big-q-Jacobi}

The \emph{big $q$-Jacobi polynomials} were introduced by Andrews and Askey \cite{AA1985}, and they are a family of basic hypergeometric orthogonal polynomials in the basic Askey scheme. For our purpose, the following specialization is required.

For each nonnegative integer $\ell$, we define a family of polynomials $\{\cJ_{\ell,n}(z)\}_{n\ge 0}$ by
\begin{align}
	\cJ_{\ell,n}(z):={}_3 \phi_{2}\left(\begin{matrix}
		q^{-n},-q^{n+\ell+1},z\\
		q^{\ell+1},0
	\end{matrix};q,q\right).
\end{align}
Here the \textit{$q$-hypergeometric series} ${}_{r+1}\phi_r$ is defined by
\begin{align*}
	{}_{r+1}\phi_{r} \left(\begin{matrix} a_1,a_2,\ldots,a_{r+1}\\ b_1,b_2,\ldots,b_r  \end{matrix}; q, z\right):=\sum_{n\ge 0} \frac{(a_1,a_2,\ldots,a_{r+1};q)_n \, z^n}{(q,b_1,b_2,\ldots,b_{r};q)_n}.
\end{align*}
It is given in \cite[p.~438, eq.~(14.5.3)]{KLS2010} that $\cJ_{\ell,n}(z)$ satisfies the three-term recursive relation
\begin{align}
	A_{\ell,n} \cJ_{\ell,n+1}(z) = (A_{\ell,n}+B_{\ell,n} -1+z) \cJ_{\ell,n}(z) - B_{\ell,n} \cJ_{\ell,n-1}(z),
\end{align}
where
\begin{align*}
	A_{\ell,n}&=\frac{1-q^{2n+2\ell+2}}{(1+q^{2n+\ell+1})(1+q^{2n+\ell+2})},\\
	B_{\ell,n}&=-\frac{q^{2n+2\ell+1}(1-q^{2n})}{(1+q^{2n+\ell})(1+q^{2n+\ell+1})}.
\end{align*}
If we normalize $\cJ_{\ell,n}(z)$ as monic polynomials
\begin{align}\label{eq:tilde-cJ-def}
	\widetilde{\cJ}_{\ell,n}(z) := \frac{(q^{\ell+1};q)_n}{(-q^{n+\ell+1};q)_n}\cJ_{\ell,n}(z),
\end{align}
then $\widetilde{\cJ}_{\ell,0}(z) = 1$, $\widetilde{\cJ}_{\ell,1}(z) = \widetilde{a}_{\ell,0}+z$, and for $n\ge 1$,
\begin{align}\label{eq:rec-tJ}
	\widetilde{\cJ}_{\ell,n+1}(z) = (\widetilde{a}_{\ell,n} + z)\widetilde{\cJ}_{\ell,n}(z) - \widetilde{b}_{\ell,n} \widetilde{\cJ}_{\ell,n-1}(z),
\end{align}
where
\begin{align*}
	\widetilde{a}_{\ell,n} &= -\frac{q^{2n+\ell+1} (1+q) (1+q^\ell)}{(1+q^{2n+\ell}) (1+q^{2n+\ell+2})},\\
	\widetilde{b}_{\ell,n} &= -\frac{q^{2n+2\ell+1} (1-q^{2n}) (1-q^{2n+2\ell})}{(1+q^{2n+\ell-1})(1+q^{2n+\ell})^2(1+q^{2n+\ell+1})}.
\end{align*}

Let us further recall a standard result for orthogonal polynomials presented in \cite[p.~25, Exercise 4.4]{Chi1978}.

\begin{lemma}
	Suppose that $\{p_n(z)\}_{n\ge 0}$ is a family of polynomials given by a three-term recursive relation for $n\ge 1$,
	\begin{align*}
		p_{n+1}(z) = (a_n + z)p_n(z) - b_n p_{n-1}(z),
	\end{align*}
	with initial conditions $p_0(z)=1$ and $p_1(z)=a_0+z$ where $\{a_n\}_{n\ge 0}$ and $\{b_n\}_{n\ge 1}$ are fixed sequences. Let
	\begin{align*}
		r_n(z) := u^{-n} p_n(uz+v)
	\end{align*}
	for $n\ge 0$ be a new family of polynomials. Then $r_n(z)$ satisfies the three-term recursive relation
	\begin{align}\label{eq:rec-r-trans}
		r_{n+1}(z) = \big(u^{-1}(a_n+v) + z\big)r_n(z) - u^{-2}b_n r_{n-1}(z).
	\end{align}
\end{lemma}

Now we introduce a family of polynomials $\{\cP_{\ell,n}(z)\}_{n\ge 0}$ for each nonnegative integer $\ell$:
\begin{align}\label{eq:cP-def}
	\cP_{\ell,n}(z):=\frac{(-1)^n (q^{\ell+1};q)_n}{q^n (1-q)^n (-q^{n+\ell+1};q)_n} {}_3 \phi_{2}\left(\begin{matrix}
		q^{-n},-q^{n+\ell+1},q\big(1-(1-q)z\big)\\
		q^{\ell+1},0
	\end{matrix};q,q\right).
\end{align}
In other words,
\begin{align*}
	\cP_{\ell,n}(z) = \frac{(-1)^n}{q^n (1-q)^n} \widetilde{\cJ}_{\ell,n}\big((q^2-q)z+q\big).
\end{align*}
The following result is a direct consequence of \eqref{eq:rec-tJ} and \eqref{eq:rec-r-trans}.

\begin{theorem}\label{th:cP-rec}
	We have $\cP_{\ell,0}(z) = 1$, $\cP_{\ell,1}(z) = a_{\ell,0} + z$, and for $n\ge 1$,
	\begin{align}\label{eq:rec-P}
		\cP_{\ell,n+1}(z) = (a_{\ell,n} + z)\cP_{\ell,n}(z) - b_{\ell,n} \cP_{\ell,n-1}(z),
	\end{align}
	where
	\begin{align*}
		a_{\ell,n} &= \frac{q^{2n+\ell} (1+q) (1+q^\ell)}{(1-q)(1+q^{2n+\ell}) (1+q^{2n+\ell+2})} - \frac{1}{1-q},\\
		b_{\ell,n} &= -\frac{q^{2n+2\ell-1} (1-q^{2n}) (1-q^{2n+2\ell})}{(1-q)^2 (1+q^{2n+\ell-1})(1+q^{2n+\ell})^2(1+q^{2n+\ell+1})}.
	\end{align*}
\end{theorem}

\section{A linear functional on $\mathbb{Q}(q)[z]$}\label{sec:lin-func}

Define for $0\le m\le n$ a family of polynomials in $\mathbb{Q}(q)[z]$,
\begin{align*}
	\qbinom{m,z}{n}_q := \frac{1}{[n]_q!} \prod_{k=m-n+1}^m ([k]_q+q^k z).
\end{align*}
It is clear that $\qbinom{m,z}{n}_q$ is of degree $n$ in $z$. Further, $\{\qbinom{n,z}{n}_q\}_{n\ge 0}$ forms a basis of $\mathbb{Q}(q)[z]$.

Let $\Phi$ be a linear functional on $\mathbb{Q}(q)[z]$ given by
\begin{align}\label{eq:Phi-n-n}
	\Phi\left(\qbinom{n,z}{n}_q\right) := \frac{1}{(-q^2;q)_n} \qquad (n\ge 0).
\end{align}

\begin{lemma}
	For $0\le m\le n$,
	\begin{align}\label{eq:Phi-m-n}
		\Phi\left(\qbinom{m,z}{n}_q\right) = \frac{(-1)^{n-m}q^{n-m}}{(-q^2;q)_n}.
	\end{align}
\end{lemma}

\begin{proof}
	It was shown in \cite[p.~19, Lemma 1]{CE2015} that for $0\le m\le n$,
	\begin{align*}
		\qbinom{m,z}{n}_q &= \sum_{k=0}^n (-1)^{n-k}q^{-n(n-m)+\binom{n-k}{2}} \qbinom{n-m}{n-k}_q \qbinom{k,z}{k}_q\\
		&=q^{-n(n-m)} \sum_{k=0}^{n-m} (-1)^{n-m-k} q^{\binom{n-m-k}{2}} \qbinom{n-m}{n-m-k}_q \qbinom{m+k,z}{m+k}_q.
	\end{align*}
	Here the $q$-binomial coefficients are given by
	\begin{align*}
		\qbinom{M}{N}_q:=\begin{cases}
			\scalebox{0.85}{%
				$\dfrac{(q;q)_M}{(q;q)_N(q;q)_{M-N}}$} & \text{if $0\le N\le M$},\\[6pt]
			0 & \text{otherwise}.
		\end{cases}
	\end{align*}
	Recalling that $\Phi$ is linear, $\Phi\left(\qbinom{m,z}{n}_q\right)$ equals
	\begin{align*}
		q^{-n(n-m)} \sum_{k=0}^{n-m} (-1)^{n-m-k} q^{\binom{n-m-k}{2}} \qbinom{n-m}{n-m-k}_q \Phi\left(\qbinom{m+k,z}{m+k}_q\right).
	\end{align*}
	In light of \eqref{eq:Phi-n-n},
	\begin{align*}
		\Phi\left(\qbinom{m,z}{n}_q\right)&= q^{-n(n-m)} \sum_{k=0}^{n-m} (-1)^{n-m-k} q^{\binom{n-m-k}{2}} \qbinom{n-m}{n-m-k}_q \frac{1}{(-q^2;q)_{m+k}}\\
		&=\frac{(-1)^{n-m}q^{\binom{m+1}{2}-\binom{n+1}{2}}}{(-q^2;q)_m} {}_2 \phi_{1}\left(\begin{matrix}
			q^{-(n-m)},0\\
			-q^{m+2}
		\end{matrix};q,q\right)\\
		&=\frac{(-1)^{n-m}q^{\binom{m+1}{2}-\binom{n+1}{2}}}{(-q^2;q)_m} \cdot \frac{q^{(m+2)(n-m)+\binom{n-m}{2}}}{(-q^{m+2};q)_{n-m}},
	\end{align*}
	thereby yielding the desired result. Here for the evaluation of the ${}_2 \phi_{1}$ series, we apply the $q$-Chu--Vandermonde Sum \cite[p.~354, eq.~(II.6)]{GR2004}
	\begin{align}\label{eq:q-Chu-Vand}
		{}_2 \phi_{1}\left(\begin{matrix}
			a,q^{-N}\\
			c
		\end{matrix};q,q\right) = \frac{a^N (c/a;q)_N}{(c;q)_N}
	\end{align}
	at the $a\to 0$ limiting case.
\end{proof}

\begin{lemma}
	For $n\ge 0$,
	\begin{align}\label{eq:Phi-n+1-n}
		\Phi\left(\qbinom{n+1,z}{n}_q\right) = \frac{1+q}{q} - \frac{1}{q(-q^2;q)_n}.
	\end{align}
\end{lemma}

\begin{proof}
	We prove by induction on $n$. It is clear that the statement is true for $n=0$ as
	$$\Phi\left(\qbinom{1,z}{0}_q\right)=\Phi(1)=\Phi\left(\qbinom{0,z}{0}_q\right)=1.$$
	Assuming that the statement is true for some $n\ge 0$, we shall show that it is also true for $n+1$. We begin by noticing that
	\begin{align*}
		&\qbinom{n+2,z}{n+1}_q - q^{n+1}\qbinom{n+1,z}{n+1}_q\\
		&= \big(([n+2]_q + q^{n+2}z) - q^{n+1}([1]_q + qz)\big) \cdot \frac{1}{[n+1]_q!} \prod_{k=2}^{n+1} ([k]_q+q^k z)\\
		&=\frac{[n+2]_q - q^{n+1}}{[n+1]_q}\cdot \qbinom{n+1,z}{n}_q\\
		&=\qbinom{n+1,z}{n}_q.
	\end{align*}
	Applying the linear functional $\Phi$ on both sides gives
	\begin{align*}
		\Phi\left(\qbinom{n+2,z}{n+1}_q\right)&= q^{n+1}\cdot \Phi\left(\qbinom{n+1,z}{n+1}_q\right)+ \Phi\left(\qbinom{n+1,z}{n}_q\right)\\
		&=q^{n+1} \cdot \frac{1}{(-q^2;q)_{n+1}} + \left(\frac{1+q}{q} - \frac{1}{q(-q^2;q)_n}\right)\\
		&=\frac{1+q}{q} - \frac{1}{q(-q^2;q)_{n+1}},
	\end{align*}
	as desired. Here we make use of \eqref{eq:Phi-n-n} together with the inductive assumption for the second equality.
\end{proof}

\begin{theorem}\label{th:Phi-relation}
	For any $P(z)\in \mathbb{Q}(q)[z]$,
	\begin{align}
		q \Phi\big(P(1+qz)\big) + \Phi\big(P(z)\big) = (1+q)P(0).
	\end{align}
\end{theorem}

\begin{proof}
	Since $\{\qbinom{n,z}{n}_q\}_{n\ge 0}$ forms a basis of $\mathbb{Q}(q)[z]$, we may write $P(z)$ as
	\begin{align*}
		P(z) = \sum_{n=0}^N c_n \qbinom{n,z}{n}_q,
	\end{align*}
	with $N$ the degree of $P(z)$ and $c_n\in \mathbb{Q}(q)$ for each $0\le n\le N$. Note that for $n\ge 0$,
	\begin{align*}
		\qbinom{n,1+qz}{n}_q = \frac{1}{[n]_q!} \prod_{k=1}^n \big([k]_q+q^k (1+qz)\big)=\qbinom{n+1,z}{n}_q.
	\end{align*}
	Therefore,
	\begin{align*}
		&q \Phi\big(P(1+qz)\big) + \Phi\big(P(z)\big)\\
		&= q \sum_{n=0}^N c_n \Phi\left(\qbinom{n+1,z}{n}_q\right) + \sum_{n=0}^N c_n \Phi\left(\qbinom{n,z}{n}_q\right)\\
		\text{\tiny (by $\substack{\eqref{eq:Phi-n-n}\\ \eqref{eq:Phi-n+1-n}}$)}&= q \sum_{n=0}^N c_n \left(\frac{1+q}{q} - \frac{1}{q(-q^2;q)_n}\right) + \sum_{n=0}^N c_n \left(\frac{1}{(-q^2;q)_n}\right)\\
		&=(1+q) \sum_{n=0}^N c_n.
	\end{align*}
	Finally, we evaluate that
	\begin{align*}
		P(0) = \sum_{n=0}^N c_n \qbinom{n,0}{n}_q = \sum_{n=0}^N c_n \frac{[n]_q!}{[n]_q!} = \sum_{n=0}^N c_n,
	\end{align*}
	thereby concluding the required relation.
\end{proof}

\begin{theorem}
	For $n\ge 0$,
	\begin{align}\label{eq:Phi-zn}
		\Phi(z^n)=\epsilon_n.
	\end{align}
\end{theorem}

\begin{proof}
	We first notice that $\Phi(z^0) = \Phi(1) = 1 = \epsilon_0$. Now for $n\ge 1$, we apply Theorem \ref{th:Phi-relation} with $P(z)=z^n$, and derive that
	\begin{align*}
		q\Phi\big((1+qz)^n\big) + \Phi(z^n) = 0,
	\end{align*}
	namely,
	\begin{align*}
		\sum_{k=0}^n \binom{n}{k} q^{k+1} \Phi(z^k)  + \Phi(z^n) = 0.
	\end{align*}
	Since this recursive relation for $\Phi(z^n)$ is identical to that for $\epsilon_n$ as given in \eqref{eq:q-Euler-rec}, we conclude that $\Phi(z^n)=\epsilon_n$.
\end{proof}

\section{Continued fractions for $\{\epsilon_{k}\}_{k\ge 0}$ and $\{\epsilon_{k+1}\}_{k\ge 0}$}\label{sec:CF}

Let $\ell\in \{0,1\}$. We define two linear functionals $\Phi_\ell$ on $\mathbb{Q}(q)[z]$ by
\begin{align}
	\Phi_\ell(z^n) := \Phi(z^{n+\ell}) \qquad (n\ge 0),
\end{align}
where the linear functional $\Phi$ is as in \eqref{eq:Phi-n-n}.

\begin{theorem}\label{th:cP-linear-func}
	Let $\ell\in \{0,1\}$. The family of monic polynomials $\{\cP_{\ell,n}(z)\}_{n\ge 0}$ given in \eqref{eq:cP-def} is orthogonal under the linear functional $\Phi_\ell$.
\end{theorem}

\begin{proof}
	The orthogonality of $\{\cP_{\ell,n}(z)\}_{n\ge 0}$ is ensured by Favard's Theorem for $\cP_{\ell,n}(z)$ satisfies a three-term recursive relation as shown in \eqref{eq:rec-P}. Note that $\Phi_0\big(\cP_{0,0}(z)\big)=\Phi_0(z^0)=\Phi(z^0)=\epsilon_0 \ne 0$ and that $\Phi_1\big(\cP_{1,0}(z)\big)=\Phi_1(z^0)=\Phi(z^1)=\epsilon_1 \ne 0$ where \eqref{eq:Phi-zn} is invoked. Hence, it suffices to show that for $\ell\in \{0,1\}$, $\Phi_\ell\big(\cP_{\ell,n}(z)\big)=0$ whenever $n\ge 1$.
	
	When $\ell=0$, we start by observing that for $k\ge 0$,
	\begin{align*}
		(q(1-(1-q)z);q)_k = (q;q)_k \qbinom{k,z}{k}_q.
	\end{align*}
	Thus,
	\begin{align*}
		\cP_{0,n}(z) = \frac{(-1)^n (q;q)_n}{q^n (1-q)^n (-q^{n+1};q)_n} \sum_{k=0}^n \frac{q^k (q^{-n},-q^{n+1};q)_k}{(q;q)_k} \qbinom{k,z}{k}_q.
	\end{align*}
	Since $\Phi_0\big(\cP_{0,n}(z)\big) = \Phi\big(\cP_{0,n}(z)\big)$, it follows that for $n\ge 1$,
	\begin{align*}
		\Phi_0\big(\cP_{0,n}(z)\big) &= \frac{(-1)^n (q;q)_n}{q^n (1-q)^n (-q^{n+1};q)_n} \sum_{k=0}^n \frac{q^k (q^{-n},-q^{n+1};q)_k}{(q;q)_k} \Phi\left(\qbinom{k,z}{k}_q\right) \\
		\text{\tiny (by $\eqref{eq:Phi-n-n}$)}&= \frac{(-1)^n (q;q)_n}{q^n (1-q)^n (-q^{n+1};q)_n} \sum_{k=0}^n \frac{q^k (q^{-n},-q^{n+1};q)_k}{(q,-q^2;q)_k} \\
		\text{\tiny (by $\eqref{eq:q-Chu-Vand}$)}&= \frac{(-1)^n (q;q)_n}{q^n (1-q)^n (-q^{n+1};q)_n} \cdot \frac{(-1)^n q^{n(n+1)} (q^{1-n};q)_n}{(-q^2;q)_n}\\
		&= 0.
	\end{align*}
	When $\ell=1$, we notice that for $k\ge 0$,
	\begin{align*}
		z(q(1-(1-q)z);q)_k = (q^2;q)_k \qbinom{k,z}{k+1}_q.
	\end{align*}
	Thus,
	\begin{align*}
		z\cdot \cP_{1,n}(z) = \frac{(-1)^n (q^{2};q)_n}{q^n (1-q)^n (-q^{n+2};q)_n} \sum_{k=0}^n \frac{q^k (q^{-n},-q^{n+2};q)_k}{(q;q)_k}\qbinom{k,z}{k+1}_q.
	\end{align*}
	Since $\Phi_1\big(\cP_{1,n}(z)\big) = \Phi\big(z\cdot \cP_{1,n}(z)\big)$, it follows that for $n\ge 1$,
	\begin{align*}
		\Phi_1\big(\cP_{1,n}(z)\big) &= \frac{(-1)^n (q^{2};q)_n}{q^n (1-q)^n (-q^{n+2};q)_n} \sum_{k=0}^n \frac{q^k (q^{-n},-q^{n+2};q)_k}{(q;q)_k} \Phi\left(\qbinom{k,z}{k+1}_q\right) \\
		\text{\tiny (by $\eqref{eq:Phi-m-n}$)}&= -\frac{(-1)^n (q^{2};q)_n}{q^n (1-q)^n (-q^{n+2};q)_n}\cdot \frac{q}{1+q^2} \sum_{k=0}^n \frac{q^k (q^{-n},-q^{n+2};q)_k}{(q,-q^3;q)_k} \\
		\text{\tiny (by $\eqref{eq:q-Chu-Vand}$)}&= -\frac{(-1)^n (q^{2};q)_n}{q^n (1-q)^n (-q^{n+2};q)_n}\cdot \frac{q}{1+q^2} \cdot \frac{(-1)^n q^{n(n+2)} (q^{1-n};q)_n}{(-q^3;q)_n}\\
		&= 0.
	\end{align*}
	The required claim therefore holds.
\end{proof}

Now we are ready to state the $J$-fraction expressions for the series generated by $\{\epsilon_{k}\}_{k\ge 0}$ and $\{\epsilon_{k+1}\}_{k\ge 0}$.

\begin{theorem}\label{th:qEuler-0-1-CF}
	Let $\ell\in \{0,1\}$. We have
	\begin{align}
		\sum_{k\ge 0}\epsilon_{k+\ell} x^k = \cfrac{\epsilon_\ell}{1 + a_{\ell,0} x -\cfrac{b_{\ell,1} x^2}{1 + a_{\ell,1} x - \cfrac{b_{\ell,2} x^2}{1 + a_{\ell,2} x - \lastddots}}}.
	\end{align}
	Further, the Favard orthogonal polynomials for $\{\epsilon_{k+\ell}\}_{k\ge 0}$ are $\cP_{\ell,n}(z)$. Here, $a_{\ell,n}$, $b_{\ell,n}$ and $\cP_{\ell,n}(z)$ are as in Theorem \ref{th:cP-rec}.
\end{theorem}

\begin{proof}
	From \eqref{eq:Phi-zn}, we know that when $\ell\in \{0,1\}$, $\epsilon_{k+\ell} = \Phi_\ell(z^k)$ for all $k\ge 0$. In view of Theorems \ref{th:cP-rec} and \ref{th:cP-linear-func}, we apply Lemma \ref{le:linear-func} and obtain
	\begin{align*}
		\sum_{k\ge 0}\epsilon_{k+\ell} x^k = \sum_{k\ge 0}\Phi_\ell(z^k) x^k = \cfrac{\Phi_\ell(z^0)}{1 + a_{\ell,0} x -\cfrac{b_{\ell,1} x^2}{1 + a_{\ell,1} x - \cfrac{b_{\ell,2} x^2}{1 + a_{\ell,2} x - \lastddots}}}.
	\end{align*}
	Here note also that $\Phi_0(z^0)=\epsilon_0$ and $\Phi_1(z^0)=\epsilon_1$.
\end{proof}

\section{Hankel determinant evaluations}\label{sec:HDet}

In this section, we complete the proof of Theorem \ref{th:q-Euler-HDet}. For \eqref{eq:Hankel:+0} and \eqref{eq:Hankel:+1}, we directly apply \eqref{eq:det-general-0} with Theorem \ref{th:qEuler-0-1-CF} in mind. For \eqref{eq:Hankel:+2}, we make use of \eqref{eq:det-general-1} so that
\begin{align*}
	\underset{{0\le i,j\le n}}{\det}(\epsilon_{i+j+2}) = \underset{{0\le i,j\le n}}{\det}(\epsilon_{i+j+1})\cdot (-1)^{n+1} \cP_{1,n+1}(0).
\end{align*}
Note that by \eqref{eq:cP-def},
\begin{align*}
	\cP_{1,n}(0)&=\frac{(-1)^n (q^2;q)_n}{q^n (1-q)^n (-q^{n+2};q)_n} {}_3 \phi_{2}\left(\begin{matrix}
		q^{-n},-q^{n+2},q\\
		q^2,0
	\end{matrix};q,q\right)\\
	&=\frac{(-1)^n (q^2;q)_n}{q^n (1-q)^n (-q^{n+2};q)_n} \sum_{k\ge 0}\frac{(q^{-n},-q^{n+2};q)_k q^{k}}{(q^2;q)_k}\\
	&=\frac{(-1)^{n+1} (q;q)_{n}}{(1-q)^n (-q^{n+1};q)_{n+1}} \sum_{k\ge 1}\frac{(q^{-n-1},-q^{n+1};q)_k q^{k}}{(q;q)_k}\\
	\text{\tiny (by $\eqref{eq:q-Chu-Vand}$)}&=\frac{(-1)^{n+1} (q;q)_{n}}{(1-q)^n (-q^{n+1};q)_{n+1}} \left(-1+(-1)^{n+1}q^{(n+1)^2}\right).
\end{align*}
Finally, invoking \eqref{eq:Hankel:+1} yields the desired result after simplification.

\section{Conclusion}\label{sec:con}

Noting that Theorem \ref{th:cP-linear-func} only covers the cases of $\ell=0$ or $1$, it is naturally asked if one could go beyond. Recall that for all $\ell\ge 0$, the orthogonality of $\{\cP_{\ell,n}(z)\}_{n\ge 0}$ is ensured by Favard's Theorem in light of \eqref{eq:rec-P}. Now let us reformulate $\cP_{\ell,n}(z)$ as
\begin{align*}
	\cP_{\ell,n}(z)= \frac{(-1)^n (q^{\ell+1};q)_n}{q^n (1-q)^n (-q^{n+\ell+1};q)_n} \sum_{k=0}^n \frac{q^k (q^{-n},-q^{n+\ell+1};q)_k}{(q^{\ell+1};q)_k} \qbinom{k,z}{k}_q.
\end{align*}
For each $\ell\ge 0$, define a linear functional $\Theta_\ell$ on $\mathbb{Q}(q)[z]$ by
\begin{align}\label{eq:Upsilon-n-n}
	\Theta_\ell\left(\qbinom{n,z}{n}_q\right) := \frac{(q^{\ell+1};q)_n}{(q,-q^{\ell+2};q)_n} \qquad (n\ge 0).
\end{align}
Then $\Theta_\ell\big(\cP_{\ell,0}(z)\big) = 1$ and for $n\ge 1$,
\begin{align*}
	\Theta_\ell\big(\cP_{\ell,n}(z)\big) &= \frac{(-1)^n (q^{\ell+1};q)_n}{q^n (1-q)^n (-q^{n+\ell+1};q)_n} \sum_{k=0}^n \frac{q^k (q^{-n},-q^{n+\ell+1};q)_k}{(q^{\ell+1};q)_k} \Theta_\ell\left(\qbinom{k,z}{k}_q\right)\\
	&= \frac{(-1)^n (q^{\ell+1};q)_n}{q^n (1-q)^n (-q^{n+\ell+1};q)_n} {}_2 \phi_{1}\left(\begin{matrix}
		q^{-n},-q^{n+\ell+1}\\
		-q^{\ell+2}
	\end{matrix};q,q\right)\\
	\text{\tiny (by $\eqref{eq:q-Chu-Vand}$)}&= \frac{(-1)^n (q^{\ell+1};q)_n}{q^n (1-q)^n (-q^{n+\ell+1};q)_n} \cdot \frac{(-1)^n q^{n(n+\ell+1)} (q^{1-n};q)_n}{(-q^{\ell+2};q)_n}\\
	&= 0.
\end{align*}
Hence, Theorem \ref{th:cP-linear-func} can be extended as follows.

\begin{theorem}
	For each nonnegative integer $\ell$, the family of monic polynomials $\{\cP_{\ell,n}(z)\}_{n\ge 0}$ given in \eqref{eq:cP-def} is orthogonal under the linear functional $\Theta_\ell$.
\end{theorem}

\begin{remark*}
	Given a family of orthogonal polynomials $\{p_n(z)\}_{n\ge 0}$ in $\mathbb{K}[z]$ and two associated linear functionals $L_1$ and $L_2$, it is clear by choosing $\{p_n(z)\}_{n\ge 0}$ as a basis of $\mathbb{K}[z]$ that there is a nonzero constant $C\in \mathbb{K}$ such that $L_1= C\cdot L_2$. In our case, we have
	\begin{align*}
		\Phi_0 = \epsilon_0\cdot \Theta_0 \qquad\text{and}\qquad \Phi_1 = \epsilon_1\cdot \Theta_1.
	\end{align*}
	On the other hand, it is a direct consequence of \eqref{eq:det-general-0} and \eqref{eq:rec-P} that for each nonnegative integer $\ell$,
	\begin{align}
		&\underset{{0\le i,j\le n}}{\det}\big(\Theta_\ell(z^{i+j})\big)\notag\\
		&\quad= \frac{(-1)^{\binom{n+1}{2}}q^{2\binom{n+2}{3}+(2\ell-1)\binom{n+1}{2}}}{(1-q)^{n(n+1)}} \prod_{k=1}^n \frac{(q^2,q^{2\ell+2};q^2)_k}{(-q^{\ell+1},-q^{\ell+2},-q^{\ell+2},-q^{\ell+3};q^2)_k}.
	\end{align}
	This identity reduces to \eqref{eq:Hankel:+0} and \eqref{eq:Hankel:+1} for
	\begin{align*}
		\underset{{0\le i,j\le n}}{\det}\big(\Phi_0(z^{i+j})\big) &= \epsilon_0^{n+1}\cdot \underset{{0\le i,j\le n}}{\det}\big(\Theta_0(z^{i+j})\big),\\
		\underset{{0\le i,j\le n}}{\det}\big(\Phi_1(z^{i+j})\big) &= \epsilon_1^{n+1}\cdot \underset{{0\le i,j\le n}}{\det}\big(\Theta_1(z^{i+j})\big).
	\end{align*}
	However, when $\ell\ge 2$, the linear functionals $\Theta_\ell$ and $\Phi$ are no longer intertwined, thereby resulting in a gap between $\Theta_\ell(z^n)$ and the $q$-Euler numbers for higher cases of $\ell$. In fact, it remains uncertain if there is a closed expression of $\Theta_\ell(z^n)$ for $\ell\ge 2$.
\end{remark*}

Interestingly, if we directly look at the normalized big $q$-Jacobi polynomials $\widetilde{\cJ}_{\ell,n}(z)$, it is possible to deduce an infinite family of nice Hankel determinant evaluations.

\begin{theorem}
	For each nonnegative integer $\ell$, define a sequence $\{\xi_{\ell,n}\}_{n\ge 0}$ by
	\begin{align*}
		\xi_{\ell,n}:=\frac{q^{(\ell+1)n} (-q;q)_n}{(-q^{\ell+2};q)_n}.
	\end{align*}
	Then
	\begin{align}
		&\underset{{0\le i,j\le n}}{\det}(\xi_{\ell,i+j})\notag\\
		& \quad= (-1)^{\binom{n+1}{2}}q^{2\binom{n+2}{3}+(2\ell+1)\binom{n+1}{2}} \prod_{k=1}^n \frac{(q^2,q^{2\ell+2};q^2)_k}{(-q^{\ell+1},-q^{\ell+2},-q^{\ell+2},-q^{\ell+3};q^2)_k}.
	\end{align}
\end{theorem}

\begin{proof}
	For each $\ell\ge 0$, we introduce a linear function $\Xi_\ell$ on $\mathbb{Q}(q)[z]$ by
	\begin{align}\label{eq:U-def}
		\Xi_\ell(z^n):=\xi_{\ell,n} \qquad (n\ge 0).
	\end{align}
	It then holds that for $n\ge 0$,
	\begin{align}
		\Xi_\ell\big((z;q)_n\big) = \frac{(q^{\ell+1};q)_n}{(-q^{\ell+2};q)_n}.
	\end{align}
	This is because the $q$-binomial theorem \cite[p.~36, eq.~(3.3.6)]{And1976} tells us that
	\begin{align*}
		(z;q)_n = \sum_{k=0}^n (-z)^k q^{\binom{k}{2}} \qbinom{n}{k}_q,
	\end{align*}
	so that
	\begin{align*}
		\Xi_\ell\big((z;q)_n\big) &= \sum_{k=0}^n (-1)^k q^{\binom{k}{2}} \qbinom{n}{k}_q \cdot \Xi_\ell(z^k)\\
		\text{\tiny (by $\eqref{eq:U-def}$)}&= \sum_{k=0}^n (-1)^k q^{\binom{k}{2}} \cdot \frac{(-1)^k q^{nk-\binom{k}{2}} (q^{-n};q)_k}{(q;q)_k} \cdot \frac{q^{(\ell+1)k} (-q;q)_k}{(-q^{\ell+2};q)_k}\\
		&= {}_2 \phi_{1}\left(\begin{matrix}
			q^{-n},-q\\
			-q^{\ell+2}
		\end{matrix};q,q^{n+\ell+1}\right)\\
		&= \frac{(q^{\ell+1};q)_n}{(-q^{\ell+2};q)_n},
	\end{align*}
	as claimed. Here we make use of the reverse $q$-Chu--Vandermonde Sum \cite[p.~354, eq.~(II.7)]{GR2004} in the last equality.
	
	Now we consider the orthogonal polynomials $\widetilde{\cJ}_{\ell,n}(z)$ as defined in \eqref{eq:tilde-cJ-def}. To show that $\Xi_\ell$ is their associated linear functional, we note that $\Xi_\ell\big(\widetilde{\cJ}_{\ell,0}(z)\big) = \Xi_\ell(1)=1$ and need to verify that $\Xi_\ell\big(\widetilde{\cJ}_{\ell,n}(z)\big) = 0$ whenever $n\ge 1$. In fact,
	\begin{align*}
		\Xi_\ell\big(\widetilde{\cJ}_{\ell,n}(z)\big) &= \frac{(q^{\ell+1};q)_n}{(-q^{n+\ell+1};q)_n}\sum_{k=0}^n \frac{q^k(q^{-n},-q^{n+\ell+1};q)_k}{(q,q^{\ell+1};q)_k} \cdot \Xi_\ell\big((z;q)_k\big)\\
		&=\frac{(q^{\ell+1};q)_n}{(-q^{n+\ell+1};q)_n} {}_2 \phi_{1}\left(\begin{matrix}
			q^{-n},-q^{n+\ell+1}\\
			-q^{\ell+2}
		\end{matrix};q,q\right)\\
		&= 0.
	\end{align*}
	
	Finally, in light of the recursive relation for $\widetilde{\cJ}_{\ell,n}(z)$ given in \eqref{eq:rec-tJ} with the same $\widetilde{a}_{\ell,n}$ and $\widetilde{b}_{\ell,n}$ therein, we have
	\begin{align}
		\sum_{k\ge 0}\xi_{\ell,k} x^k = \cfrac{1}{1 + \widetilde{a}_{\ell,0} x -\cfrac{\widetilde{b}_{\ell,1} x^2}{1 + \widetilde{a}_{\ell,1} x - \cfrac{\widetilde{b}_{\ell,2} x^2}{1 + \widetilde{a}_{\ell,2} x - \lastddots}}}.
	\end{align}
	Applying \eqref{eq:det-general-0} yields the desired Hankel determinant evaluations.
\end{proof}

\subsection*{Acknowledgements}

S.~Chern was supported by a Killam Postdoctoral Fellowship from the Killam Trusts.

\bibliographystyle{amsplain}

\end{document}